\documentclass{article}
\usepackage{amsmath,amsthm,amssymb,amstext} 
\usepackage{mathtext}
\usepackage[T1,T2A]{fontenc} 
\usepackage[utf8]{inputenc}
\usepackage{url}
\usepackage{comment}
\usepackage[english,russian]{babel}

\newtheorem{corollary}{Corollary}
\newtheorem{definition}{Definition}
\newtheorem{lemma}{Lemma}
\newtheorem{theorem}{Theorem}
\newtheorem*{Ntheorem}{Theorem}
\newtheorem{remark}{Remark}

\DeclareMathOperator{\suminf}{\mathbin{\between\!\!\!\sum}}
\def\Ykase{\foreignlanguage{russian}{Y}} 

\begin{document}
\selectlanguage{english}

\title{Attempting to remove infinites from divergent series: Hardy will hardly help.}

\author{Mario Natiello\thanks{Centre for Mathematical Sciences,
    Lund University, Box 118, S-22100, Lund, Sweden, {\tt Mario.Natiello@math.lth.se} (corresponding
    author).}\\ and \\
Hernán G Solari\thanks{Departamento de Física, FCEN-UBA e
    IFIBA-CONICET, Argentina, {\tt solari@df.uba.ar}}}

\maketitle

\begin{abstract}
The consequences of adopting other definitions of the concepts of sum and
convergence of a series are discussed in the light of historical and
epistemological contexts.  We show that some divergent series appearing in the
context of renormalization methods cannot be assigned finite values in a form
consistent with Hardy's axioms without at the same time equating one to zero,
thus destroying the mathematical building. We show that if the replacements for
the concept of sum of a series are required to be associative, to be invariant
under finite permutations of the terms and dilution, further restrictions
emerge. We finally discuss the epistemological costs of accepting these
practices in the name of instrumentalism.
\end{abstract}


\section{Introduction}

The possibility of assigning a finite value to divergent series has
recently made it to the news\footnote{See the New York edition of New York
  Times, page D6 of February 4, 2014, on-line at \\ 
\url{http://www.nytimes.com/2014/02/04/science/in-the-end-it-all-adds-up-to.html?smid=fb-share&_r=1}} in a way that is unusual for science or 
mathematics news. Indeed, this news and even some related Youtube videos 
seem to lie halfway between joke and serious stuff, but in the end it 
turns out that these contributions are intended (by their authors) to 
be serious. 

It is then compelling to assess the scientific relevance of these
methods and in particular the issue of (logical) consistency of these
methods in relation to the body of mathematical knowledge, as well as
their epistemological implications.

In Section \ref{bgr} we display the mathematical background of
the problem. Section \ref{stat} contains the statement of the main results of
this work along with their proofs. Further, we discuss historical and 
epistemological issues in Section \ref{disc} while Section \ref{conc} is 
devoted to conclusions. 

\section{Background on Series}\label{bgr}

\subsection{Sum of a Series}
We follow here standard textbooks \cite[p. 383-4]{apos67} reviewing
the definition of series and some of its basic properties, focusing
on what will be relevant for the coming Sections.

Consider a sequence $\{a_0,a_1,a_2,\cdots\}$ and also its sequence
of partial sums $\displaystyle S_N=\sum_{k=0}^N a_k$. This sequence
of partial sums is called an {\em infinite series}, or simply {\em series}.

\begin{definition}\label{basic}
If there exists a number $s$ such that  $s=\lim_{N\to\infty}S_N$ we 
say that the series $\displaystyle\sum_{k=0}^\infty a_k$ is convergent 
with sum $s$ and write $\displaystyle\sum_{k=0}^\infty a_k=s$. Otherwise, 
we say that the series is {\em divergent}.
\end{definition}

We may further distinguish among divergent series those where the
$\lim_{N\to\infty}S_N$ is $+\infty$ or $-\infty$ on the one hand
(we may call them {\em divergent to infinity}) and
those where this limit does not exist.

\begin{definition}\label{conv}
A convergent series $\displaystyle\sum_{k=0}^\infty a_k$ 
is called {\em absolutely convergent} if the series 
$\displaystyle\sum_{k=0}^\infty |a_k|$ is convergent. Otherwise, 
it is called {\em conditionally convergent}.
\end{definition}

\subsubsection{Basic Properties}
We follow an approach inspired on the posthumous book by 
Hardy\cite{hard49} on divergent series.
For convergent series, the following properties can be demonstrated
as theorems:
\begin{itemize}
\item[(A)] For any real $k$, 
           $\displaystyle \sum_{n}a_{n}=s\Rightarrow\sum_{n}(ka_{n})=ks$.
\item[(B)] $\displaystyle\sum_{n}a_{n}=s$ and 
           $\displaystyle\sum_{n}b_{n}=r\Rightarrow\sum_{n}(a_{n}+b_{n})=s+r$.
\item[(C)] $\displaystyle\sum_{n=0}a_{n}=s\Leftrightarrow\sum_{n=1}a_{n}=s-a_{0}$.
\end{itemize}
The proofs are a direct application of the properties of the limit of 
a sequence (in this case the sequence of partial sums).
\begin{remark}
Properties (A) and (B) are recognised as {\em linearity} and (C) is
called {\em stability}. They extend the natural properties of sums
for the sequence of partial sums $S_N$ all the way through the limit.
\end{remark}
An immediate generalisation of (C) to finitely many operations is the 
following
\begin{corollary}
\label{c1} For any positive integer $N$, 
\begin{equation*}\sum_{n=0}a_{n}=s
    \Leftrightarrow\sum_{n=N}a_{n}=s-\sum_{k=0}^{N-1} a_k.
\end{equation*}
\end{corollary}

We list here other natural properties of convergent series 
extrapolated from finite sums 
via the limit properties for the sequence of partial sums.
\begin{corollary}\label{CD} (a) \textbf{Associativity} If the series 
$\sum_{n}a_{n}$ has a (finite or infinite) sum, then the series 
$\sum_{k}b_{k}$ obtained via $b_{k}=a_{2k}+a_{2k+1}$ for some or all
nonnegative integers $k$, has the same sum.\\
(b1) \textbf{Commutativity} If the series
$\sum_{n}a_{n}$ has a (finite or infinite) sum then the series
$\sum_{n}b_{n}$ obtained via $(b_{2k},b_{2k+1})=P(a_{2k},a_{2k+1})=(a_{2k+1},a_{2k})$
for some or all nonnegative integers $k$, has the same sum ($P$
is the nontrivial permutation of $2$ elements).\\
(b2) For series having the commutativity property, 
finite compositions of permutations of order up to $N$ (where $N$ a positive 
integer) do not alter the sum of the series.\\
(c) \textbf{Dilution} If the series with elements $a_{0},a_{1},a_{2},\cdots$
has a (finite or infinite) sum then the series with elements 
$a_{0},0,a_{1},0,a_{2},0,\cdots$ 
i.e., inserting a zero between some or all pairs of elements in the original
sequence, has the same sum. \end{corollary} 
\begin{proof}For associativity, collecting up to $N$ terms corresponds to picking a 
subsequence from the original sequence of partial sums, having thus the
same limit. For (b1) every other partial sum coincides with the original
ones. For (b2), the partial sums of the new series coincides with the original
one every $N$ steps. In between, they differ at most in a finite sum of terms 
that goes to zero for $k\to\infty$. Hence, both sequences of partial sums 
have the same limit. As for 
dilution, if one takes a convergent sequence $\{S_N\}$ and duplicates its 
terms: $S_1, S_1, S_2, S_2, \cdots$, the new sequence has the same limit 
as the original one. 
Hence, dilution does not alter the sum of the series.
\end{proof}

The above properties can be arbitrarily (but finitely) combined, without
altering the sum of a series. However, more drastic rearrangements need not
preserve the sum unless the series is absolutely convergent.
In fact, any {\em rearrangement} of an
absolutely convergent series produces a new series with the same sum as the
original one. However, invariance in front of arbitrary rearrangements of terms
does not hold for conditionally convergent series. This is the content of {\em
Riemann rearrangement theorem} \cite[p. 413]{apos67}.
\begin{Ntheorem}[Riemann] Let $\sigma(n)$ be an injective function of the
positive integers and $K$ some real number. Suppose that
$\{a_1,a_2,a_3,\cdots\}$ is a sequence of real numbers, and that
$\sum_{n=1}^\infty a_n$ is conditionally convergent. Then there exists a {\em
rearrangement} $\sigma(n)$ of the sequence such that $\sum_{n=1}^\infty
a_{\sigma (n)} = K$.  The sum can also be rearranged to diverge to $\pm\infty$
or to fail to approach any limit.\end{Ntheorem}

\begin{remark}
A classical example of rearrangement is the alternating series
$\displaystyle a_n=\{\frac{(-1)}{n}^{n+1}\}$, $n\ge 1$.  This series sums to
$s=\ln{2}$. Create a new series by dilution
adding one zero before each element of the series and dividing by $2$, i.e.,
$\displaystyle \{0,a_1/2,0,a_2/2,\cdots\}$. Combine then this series and the
original one as in property (B). The new series has sum $s=\frac{3}{2}\ln{2}$.
However, after disregarding the intermediate zeroes, the combined series is a
rearrangement of the original one, where the negative terms (for $n=2k$) appear
every third element instead of every other element. The new series adds two
positive numbers for every negative contribution, thus subtracting the 
negative contributions in a different way from that in the original series.
\end{remark}
\begin{remark}
All conditionally convergent series can be decomposed into two monotonic 
series: one with the positive terms, diverging to $+\infty$ and another 
with the negative terms, diverging to $-\infty$. 
What Riemann's Rearrangement Theorem teaches us is that if one wishes to 
interpret the sum of such a series pictorially as the outcome from 
``cancellation of both infinities'', then there is actually not one way 
to do it, but infinitely many, depending of the order in which the 
elements of the two participant series are picked up. However, each 
possible result is the unique limit of a specific sequence of partial 
sums.
\end{remark}

\subsection{Signification}

Mathematics can be regarded as a process of successive abstractions originating
in real-life situations. Thus, natural numbers abstract the process of counting
objects and involve the abstraction of the concept of addition as well. Instead
of saying: {\em one goat and another goat makes two goats and another goat
makes three goats \dots} and the same for all indivisible objects, we say
$1+1+1\dots$ abstracting away all objects and {\em dressing} the final answer
with them again (I counted $n$ {\em goats}). We increase our ability of
counting moving to objects with more elaborated properties (e.g., divisible
objects or portions of objects, debts, missing objects), producing the realms
of integers, rationals, etc. The properties of addition are {\em extended}
(i.e., preserving its original properties) to these higher levels of
abstraction. The operations of abstraction and its inverse dressing relate
mathematics to the material world. We call them more precisely abstraction and
{\em signification}. 

The process of abstraction is also known as {\em idealization} and in physics
is historically linked with Galileo Galilei and his discussion of free fall.
\cite[pp. 205]{gali38} (for an English translation see \cite[pp. 170]{gali14})
Insight on the process of signification can also be traced back to Galileo's words 
announcing that mathematics is the language of the universe \cite{gali23}, 
thus recognising mathematics as belonging to the realm of the material
world.

Hence, when addressing issues of the material world and its sciences,
mathematical objects retain a specific signification. Any new mathematical
object introduced along the investigation is related by abstraction and
signification on one hand to the material world, on the other hand to
mathematics where the object becomes context independent by the very process of
abstraction. Thus, when counting apples we use the same addition than when we
count goats, nodes in a vibrating string or smiling faces.
The reciprocally inverse processes of {\em abstraction} and {\em signification}
lie in the foundations of any attempt to understand the material world with
mathematical tools, and cannot be disrupted nor neglected in any of its parts.

\subsection{Minimal Requirements} Let us suppose that we want to resign
Definitions \ref{basic} and \ref{conv} and produce another definition related
to the sequence $\{a_0,a_1,a_2,\cdots\}$.  We require however to keep as much
as possible of the original properties of series summation. We raise then
properties (A-C) to the status of axioms, while replacing the series symbol by
something new, since that symbol received its meaning in Definition \ref{basic}
which we are now resigning.  Inspired by Hardy in \cite{hard49}, we assume
hereafter that axioms (A-C) hold for the new object $\suminf_n a_n$. Each new
method of assigning values to infinite sequences should provide its own
definition (as well as signification) for this object. Corollary \ref{c1} now
reads: $\displaystyle \suminf_{n=0}a_{n}=s
\Leftrightarrow\suminf_{n=N}a_{n}=s-\sum_{k=0}^{N-1} a_k$. The last term of the
RHS is the plain sum of a finite set of numbers.

We will distinguish those methods that sum convergent series and series 
diverging to infinity in the usual way, namely,

\begin{definition}[(see p. 10 in \cite{hard49})]\label{reg} (a) 
A method\footnote{The symbol {\Ykase} is inspired in the Cyrillic word
\foreignlanguage{russian}{указ} meaning decree or edict (formally: {\em imposition}).}% 
 $\Ykase$ assigning a finite
value to a series is called {\em regular} if this value coincides
with the standard sum in the case of standard convergent series.\\
(b) A regular method $\Ykase$ is called {\em totally regular} if series
diverging to $\pm\infty$ with the standard definition 
also diverge to $\pm\infty$ in $\Ykase$. \end{definition} 

Axioms (A-C) may be regarded as a minimal compromise. Any method which is {\bf
not} linear and stable cannot be seriously considered as an alternative to the
sum of a series (it would be either not linear or not finitely related to
ordinary sum). Also, in Hardy's view regularity is a minimal requirement:
whichever method not complying with standard results for standard convergent
problems cannot be seriously considered as an extension of the concept of sum.
Already in Section 1.4 of Hardy's book 
it is noted that the method $\mathfrak{E}$ of
analytic extension to assign values for divergent series 
is not regular.
We will comment on this method in Subsection \ref{sierp}.

\section{Statement of Results}\label{stat}

We state now the main results of this work:
\begin{theorem}\label{T1} Any method $\Ykase$ assigning a finite number to 
$\displaystyle\suminf_n 1$ (proposed replacement for $1+1+1+1+1+\cdots$)
is (i) \textbf{not totally regular}, (ii) \textbf{not regular} and 
(iii) \textbf{contradictory}. \end{theorem}

\begin{theorem}\label{T2} Any method $\Ykase$ assigning a finite number to 
$\displaystyle\suminf_n n$ (proposed replacement for $1+2+3+4+5+\cdots$) 
is (i) \textbf{not totally regular}, (ii) \textbf{not regular} and 
(iii) \textbf{contradictory}. \end{theorem}

By {\em contradictory} we mean that incompatible statements
corresponding to $r=\Ykase(\{a_n\})=s$ for
different real numbers $r$ and $s$ can be proved in this context. 

It goes without saying that we are speaking about well-posed methods $\Ykase$
where the assigned values for $\suminf_n a_n$ are unique, whenever the
sequence belongs to the domain of the method. 
Also, the definition of regularity naturally assumes that the sequences
associated to all convergent series belong to the domain of whichever 
method $\Ykase$ is under consideration (even non-regular ones).

\subsection{Proofs}

\begin{proof}[Proof of Theorem \ref{T1}] That the method $\Ykase$ is not
totally regular is immediate since otherwise it should assign the value
$+\infty$ to the proposed expression. By (C), $\suminf_n
1\equiv\Ykase(\{1,1,1,\cdots\})$ has the same value as $0+1+1+1+\cdots$ (namely
$\Ykase(\{0,1,1,\cdots\})$) and hence the term-wise difference of both objects
by (B) satisfies: $1+0+0+0+\cdots\equiv\Ykase(\{1,0,0,0,\cdots\})=0$. Since
$1+0+0+0+\cdots$ is a convergent series with sum $1$ with the standard
definition, we conclude that the method $\Ykase$ is not regular. As for
contradiction, using (A) we obtain in the same way:
$-1+0+0+0+\cdots\equiv\Ykase(\{-1,0,0,0,\cdots\})=0$ thus establishing by (C)
that $1=\Ykase(\{0,0,0,0,\cdots\})=-1$ since both those numbers can be assigned
as value for $0+0+0+0+\cdots$, which by (C) belongs to the domain of the method
$\Ykase$. This result is contradictory with the whole body of mathematics since
obviously $1\ne -1$.  \end{proof}

\begin{proof}[Proof of Theorem \ref{T2}] That the method $\Ykase$ is
not totally regular is immediate since otherwise it should assign the value
$+\infty$ to the proposed expression. As for regularity and consistency, we will
prove that $\Ykase(\{1,1,1,1,\cdots\})$ 
belongs to the domain of $\Ykase$ and compute its value.
Let $s$ be the value of $1+2+3+4+5+\cdots$, i.e.,
$\Ykase(\{1,2,3,4,\cdots\})=s$. Then by (C),
$\Ykase(\{0,1,2,3,4,\cdots\})=s+0=s$ and also by (A),
$\Ykase(\{-0,-1,-2,-3,-4,\cdots\})=-s$. Hence, we obtain by (B):
$1+1+1+1+\cdots\equiv\Ykase(\{1,1,1,\cdots\})=0$.  Hence, this
expression belongs to the domain of the method $\Ykase$ having the value 
zero and the
previous theorem applies. Moreover, repeated application of (C) on 
$1+1+1+1+\cdots\equiv\Ykase(\{1,1,1,\cdots\})=0$, results in
$1+1+1+1+\cdots\equiv\Ykase(\{1,1,1,\cdots\})=-N$ for any
nonnegative integer $N$. This result provides an alternative proof of
contradiction in Theorem \ref{T1}.
\end{proof}

\subsection{Associativity, Commutativity and Dilution}

Let us now consider $\suminf_{n=0}(-1)^{n}$. Any method assigning a value $s$
to it should comply $s=1-s$ and hence $s=1/2$.  There exist many totally
regular methods for the purpose, the most famous of which is probably the {\em
Cesaro sum}, defined as the limit of the the sequence of successive averages of
partial sums, i.e., letting $Z_{n}=(\frac{1}{n})\sum_{k=0}^{n-1}S_{k}$ for
$n\ge 1$, (being $S_{k}$ the partial sum of the first consecutive elements of
the original series up to $k$) we have $Z_{n}=(\frac{1}{2})+(\frac{c}{n})$,
where $c=0$ or $1$ and the Cesaro sum of $\suminf_{n=0}(-1)^{n}$ is
$\lim_{n\to\infty}Z_{n}=1/2$. 
 
\begin{lemma}\label{Cesaro}Cesaro sums fulfill {\bf none} of the properties 
associativity, commutativity and dilution.\end{lemma}

\begin{proof} For associativity just note that summing the elements 
of $\suminf_{n=0}(-1)^{n}$ pairwise
we obtain either $0+0+0+\cdots$ or $1+0+0+\cdots$ (starting the association in 
the first or in the second element of the original series), both having
different Cesaro sums and both different from $1/2$.
For commutativity, permute the elements
$(a_{n},a_{n+1})$ for all odd $n$, obtaining $1+1-1-1+1+1-1-1+\cdots$
whose Cesaro sum is unity. For dilution, insert one zero only
after each positive element, obtaining the Cesaro sum $2/3$.
\end{proof}

\begin{lemma}\label{T3} Any commutative method $\Ykase$ assigning a (finite)
value to $\suminf_{n=0}(-1)^{n}$ is contradictory. \end{lemma} \begin{proof}
Whichever commutative method $\Ykase$ is adopted, (C) forces
$\suminf_{n=0}(-1)^{n}=r=1-r$ and hence $r=1/2$.  However, by permuting the
elements pairwise, from
$\suminf_{n=0}(-1)^{n}\equiv\Ykase(\{1,-1,1,-1,\cdots\})=r$ we obtain
$\suminf_{n=0}(-1)^{n+1}\equiv\Ykase(\{-1,1,-1,1,\cdots\})=r$. Also, the first
expression by (A) and multiplication by $-1$ leads to
$\suminf_{n=0}(-1)^{n+1}\equiv\Ykase(\{-1,1,-1,1,\cdots\})=-r$. We obtain the
contradiction  $r=\Ykase(\{-1,1,-1,1,\cdots\})=-r$ with $r=1/2$.\end{proof} 
 
\subsection{Euler's Continuation Method $\mathfrak{E}$ \cite{hard49}}\label{sierp}
In Hardy's account of possible alternatives to standard summation, Abel's
method and Euler's continuation method $\mathfrak{E}$ are 
highlighted \cite[p. 7]{hard49} (among others). We may say that Euler glimpsed 
a possible approach in a moment in history where the concept of convergence
was not fully developed (see next Section) while Abel formulated the idea with 
exhaustive precision. In fact, Abel's second theorem \cite[p. 3]{bore28} can be stated as:
\begin{Ntheorem}[Abel] 
If $\displaystyle\sum_n a_n=s$ and 
$\displaystyle\lim_{x\to 1^-}\sum_n a_nx^n=S$, then $s=S$.\end{Ntheorem}
In other words, if the series is convergent, with finite sum $s$, and if 
the power series has a (finite) limit $S$, then both numbers coincide.

Hardy's account of $\mathfrak{E}$ is that if 
$\sum_n a_nz^n$ defines an analytic function $f(z)$ in some region of the 
complex plane such that the function is properly defined along a path from 
that region up to $z=1$, then $\sum_n a_n=f(1)$. $\mathfrak{E}$ is indeed
less precise than Abel's theorem. Let us consider in which ways the assumptions
of Abel's theorem may fail in the context of $\mathfrak{E}$. Either $\{a_n\}$ 
is a divergent series or $\lim_{z\to 1}\sum_n a_nz^n$ does not exist 
(now with $z\in\mathbf{C}$).

For the first case, there are examples of power series having exactly 
the same shape in different regions of the complex plane, while defining
different functions. Then there is no uniquely defined ``$f(z_0)$''. For 
an example with the functions 
$f(z)=\displaystyle\pm\frac{1+z^2}{1-z^2}$ and $z_0=2i$ see 
\cite[p.16]{hard49}. Another example could be the geometric series 
for $f(z)=(1-z)^{-1}$ contrasting the series inspired in Riemann's zeta 
function: $g(z)=\sum_{n\ge 1} n^{1-z}$. Both series
have disjoint domains of convergence in the complex plane. For $z=1$
both series -in the context of $\suminf$ or $\Ykase()$- could be taken to 
represent $1+1+1+1+\cdots$ but the second function gives a finite 
number while the first diverges.

For the second case, Sierpinski \cite{sier16} gives an example of a
convergent series $\sum_n a_n$ and a power series derived from it 
(defining a function $f(z)$) such that 
while the power series has a limit for $x\to 1^-$ along the real axis 
(fully compatible with Abel's theorem), the limit does not exist along 
arbitrary paths $z\to 1$ in the complex plane. This has more than anecdotic
value, since Sierpinski series combined with $\mathfrak{E}$ could be used 
to ``destroy convergence'' in any convergent series. Let $r=\sum_n a_n$ be the 
sum of Sierpinski's series. Let $\sum_n b_n$ be a convergent series and 
consider the expressions $r-r+\sum_n b_n$ and $f(z)-r+\sum_n b_nz^n$.
Using Abel's method, both expressions have the same value, namely $\sum_n b_n$. 
Using $\mathfrak{E}$, the first expression still yields the same value, while 
the second one diverges.

Clearly, $\mathfrak{E}$ cannot be used as a tool in this context without
further and accurate specification. On the contrary, following Borel
\cite{bore28} we may say that Abel's approach is exhaustive and there is no 
room for improvement. Different attempts to prove the converse of Abel's
theorem after adding adequate additional hypotheses, have originated the
branch of mathematics called Tauberian theorems.

\section{Discussion}\label{disc}

\subsection{Historical Digression}

Extending on Hardy's account, it is to be noted that the modern concept
of limit was established by Cauchy around 1821. However, he could not
solve the question of uniform convergence. In fact, it is said \cite{laka76} 
that this issue worried Cauchy to the point of never publishing the second 
volume of his course of analysis, nor consenting to a reedition of the 
first. He eventually allowed the publication of the lecture notes of 
his classes by his friend and student Moigno in 1840\cite{moig40}. 
Again according to Lakatos, the distinction between point convergence and 
uniform convergence was unraveled by Seidel in 1847\cite{laka76}, 
thus completing the approach of Cauchy. The modern way of regarding 
limits and convergence could be said to originate around 1847.

\subsection{Epistemological Issues}

The idea of substituting a definition with another one is not free from
consequences. Definitions in mathematics may look arbitrary at a first glance
but they are always motivated. Fundamentally, (a) they satisfy the need of
filling a vacancy of content in critical places where precision is needed
(however, since many textbooks present definitions without discussing the
process for producing them, the epistemological requirements remain usually
obscure) and (b) they are explicitly forbidden to be contradictory or logically
inconsistent with the previously existing body of mathematics on which they
rest. In addition, when dealing with the mathematisation of natural sciences,
definitions carry a signification, which is the support for using that
particular piece of mathematics in that particular science.

While we appreciate the exploration work around concepts that has been done
over the years, we do not substitute a meaningful and established concept with
something that is inequivalent to it in the common domain of application.
Again, when understanding natural sciences, such substitution would disrupt the
signification chain. In simple words, we do not replace a meaningful content
with a meaningless one. This would be to depart from rationality, something
that is positively rejected by mathematics as a whole as well as by science in
general and by a large part of society.

Along the presentation, we cared to put limits to the 
possible relation between ${\Ykase}(\cdot)$ and ordinary sums.
In the light of the proven Theorems, it is
verified that such relation is feeble or absent. Hence, the very inspirational
root of these techniques becomes divorced from its results and effects. As
stated above, the alternative of resigning one of axioms (A-C) also destroys
any possible relation to ordinary sums.

Replacements that assign $\suminf n = a$ or 
$\suminf 1=b$ (with $a,b$ real numbers)
destroy the basis of mathematics, making it the same to have one goat that
having a million goats. 
We must emphasize that regularity is a necessary condition to preserve
signification but it is not sufficient. Whatever replacement we attempt must
provide a rationale for the method, preserving signification within
mathematics (in the chain of abstractions it belongs) and in relation to
natural sciences. For the case of series, signification is further destroyed 
along with properties such as association, permutation and dilution.
The alternative of using one or another definition depending of the matter
under study simply destroys the role mathematics as a whole. 
Instead of having {\em eternal and pure relations
accessible by reason alone} \cite{plato60}, it will turn mathematics to be
dependent of the context of use.

\subsubsection{The Epistemology of Success}

The issue of assigning a finite value to divergent series with methods that are
not regular and are contradictory under Hardy's axioms is not only material of
newspaper notes, discussion blogs or Youtube videos. It has actually reached
the surface of society from stuff published as
scientific material.
We support this statement by commenting on a couple of references.
This issue is not just a feature of these two citations, but
the standard procedure of a community: just read the references in
\cite{birr82} to find a large amount of practitioners of this community.

In \cite{nest97} we encounter an attempt to justify the use of the Riemann's
Zeta function. The authors refer to Hardy's book for the actual method. 
They use axioms A and B and the zeta function to write
equality between $\sum_{n=1}^\infty n$ and $-1/12$ (see their eq. (2.20)).
The conclusion is evident: the method does not comply with Hardy's axioms.
Furthermore, the result is false since to reach their conclusion the authors 
disregard a divergent contribution. Hence, the equal sign does not relate 
identical quantities as it should. The correct expression would be
\begin{equation}
\suminf_{n=1}^\infty n\equiv\Ykase(\{1,2,3,4,\cdots\})=-\frac{1}{12}
\end{equation}
where $\Ykase$ must be understood as the method based on Riemann's Zeta
function. Here, Theorem \ref{T2} applies.
 
Our second example, concerning Theorem \ref{T1},
is the book \cite{birr82} where on page 167 we
read ``The analytic continuation method converts a manifestly infinite series
into a finite result'' exemplifying with 
$\suminf_{n=1}^\infty 1\equiv\Ykase(\{1,1,1,\cdots\})=1/2$ (our notation, the 
authors use standard series notation) using the same procedure as \cite{nest97}. 
On p. 165 this expression was given the value $-1/2$, probably a typo. 
Needless to say, the authors do not use the symbol $\Ykase$ but they refer to
the method as a 
``formal procedure'' which leaves the matter in the ambiguities of language. 
If by formal we read {\em belonging to or constituting the form or essence of 
a thing}, we strongly disagree, since {\em essence} is the result of an 
abstracting (usually analytic) procedure \cite{hege71}. 
However, if ``formal'' is intended as in its second accepted meaning:
{\em following or according with established form, custom, or
rule}, we agree, observing that such social agreements are not a part of
science.

In defense of such procedures it is usually said that the theories using them
are among the most precise and successful in Physics.  This argumentation
claims, then, that questions of unicity of results, backward compatibility of a
method with standard convergent series, or its relation to sums (let alone
signification) are uninteresting. The value is assigned because in such a way
one obtains the ``correct'' result.  Hence they adhere to a (false)
epistemology that Dirac called {\em instrumentalism} \cite[page 185]{krag90}
and we plainly call {\em the epistemology of success}.

Hitting (what is claimed to be) the right answer is not equivalent to using the
right method. One may hit a correct answer with a wrong method just by chance,
by misunderstanding, or even by adaptation to the known answer, etc.
Paraphrasing Feyerabend's {\em everything goes}: an idea may be welcome as a
starting point without deeper considerations (within ethical limits, of
course). However, for that idea to be called {\bf scientific} it has to comply
with the scientific method.  Moreover, it has to comply with rationality.  The
attitude described by {\em something is right because it gives the correct
answer} is dangerous in many levels. The mathematical attitude is actually the
opposite (and logically inequivalent to it): {\em If it gives the wrong answer,
either the assumptions or the method in use are incorrect}\cite{popp59}. This
holds also for natural sciences, where in addition we have, through
signification, a safe and independent method to distinguish wrong answers from
right answers. Note that independency is crucial. Natural science makes
predictions that can be tested independently of the theory involved. If
verified, they give continued support to the theory, while if refuted they
indicate where and why to correct it. As a contrast, a theory making
predictions that can only be tested within itself obtains at best internal
support for being consistent, but it never speaks about Nature since
predictions are not independently testable.  In any case, having the right
answer is not a certificate of correctness (there may be an error somewhere
else) whereas having {\bf any} wrong answer is a certificate of incorrectness
(the error is ``there'').

It is worth to keep in mind the attitude taken by the founding fathers of
Quantum Electrodynamics
\begin{quote}
The shell game that we play [...] is technically called
'renormalization'. But no matter how clever the word, it is still
what I would call a dippy process! Having to resort to such
hocus-pocus has prevented us from proving that the theory of quantum
electrodynamics is mathematically self-consistent. It's surprising
that the theory still hasn't been proved self-consistent one way or
the other by now; I suspect that renormalization is not
mathematically legitimate. {\bf Richard Feynman}, 1985\cite{feyn83}
\end{quote}

\begin{quote} I must say that I am very dissatisfied with the situation,
because this so called good theory does involve neglecting infinities which
appear in its equations, neglecting them in an arbitrary way. This is just not
sensible mathematics. Sensible mathematics involves neglecting a quantity when
it turns out to be small - not neglecting it just because it is infinitely
great and you do not want it! {\bf Paul Dirac}. \cite[page 195]{krag90}
\end{quote}

\section{Concluding Remarks}\label{conc}

Definitions are not arbitrary, any extension of an established operation (such
as addition) needs to preserve the properties of the operation when applied to
objects in the original domain of definition (regularity). We have shown in
this sense that some methods proposed as extensions for the sum of convergent
series (standard definition) fail on this regard. An important example of this
failure is the method based on Riemann's Zeta function that has been associated
to renormalization procedures in some branches of physics. Moreover, regularity
is not enough, the extension needs to preserve all the mathematical
building-blocks it rests on.

It is important to indicate that attempts to justify renormalization such as
\cite{nest97} ought to be considered scientific attempts on the ground of
Popper's demarcationism\cite{popp59} since by connecting to mathematical 
subjects as the problem of infinite series they offer the rationale behind 
their procedures of scientific enquire.
The result of the examination indicates that such replacements must be rejected.

The decision to expose the issues concerning the mathematical foundations of
these matters to the general public should also be commented.  The
understanding of scientific matters is not reserved to an elite of
practitioners that guard the ``truth'' of the subject as priests of a cult.
Opening science to the scrutiny of the general public, including scientists
outside the paradigm is simply correct, as discussed by Lakatos \cite{laka76}.

The following thesis should be considered: the more than 50 years that this
matter has stayed without resolution is a demonstration that in closed elitist
communities the interest of (return for) the community may very well have
priority over the public (humane) interest.

\section*{Acknowledgments} We thank Alejandro Romero Fernández for his
valuable support in historical-philosophical issues. We thank Frank
Wikstr{\"o}m for explaining and discussing issues around Sierpinski's article.
The authors have not received (nor expect to receive) economic or symbolic
remuneration for the present work. It has been produced only because of
their love to reason.
\bibliographystyle{plain}
\bibliography{referencias}

\end{document}